\documentclass[11pt]{amsart}
\usepackage[margin=1.08in]{geometry}

\usepackage{amssymb}
\usepackage{microtype}
\microtypesetup{expansion=false}
\usepackage{enumitem}
\usepackage{xcolor}
\definecolor{AMSDarkBlue}{RGB}{0,51,102}
\usepackage{hyperref}
\hypersetup{
  colorlinks=true,
  linkcolor=AMSDarkBlue,
  citecolor=AMSDarkBlue,
  urlcolor=AMSDarkBlue,
  bookmarksnumbered=true,
  bookmarksopen=true,
  bookmarksopenlevel=2,
  bookmarksdepth=3,
  breaklinks=true,
  pdfpagemode=UseOutlines,
  pdftitle={A complete classification of permutation binomials of the form X\string^r(X\string^(q-1)+a) over finite fields},
  pdfsubject={Permutation binomials over finite fields},
  pdfauthor={Xiang Fan},
  pdfkeywords={permutation polynomials, permutation binomials, finite fields, Hermite's criterion, Lucas' theorem, Farey sequences}
}
\numberwithin{equation}{section}

\newtheorem{theorem}{Theorem}[section]
\newtheorem{proposition}[theorem]{Proposition}
\newtheorem{lemma}[theorem]{Lemma}
\newtheorem{corollary}[theorem]{Corollary}
\allowdisplaybreaks[1]
\title[Classification of permutation binomials]
{A complete classification of permutation binomials of the form
$X^r(X^{q-1}+a)$ over finite fields}
\author{Xiang Fan}
\address{School of Mathematics, Sun Yat-sen University, Guangzhou 510275, China}
\email{fanx8@mail.sysu.edu.cn}
\date{}

\begin{document}
\begin{abstract}
We classify, for every prime power \(q\) and every \(e\geqslant2\), the
permutation binomials \(X^r(X^{q-1}+a)\) over \(\mathbb F_{q^e}\).  Writing
\(\ell_j(q)=(q^j-1)/(q-1)\), such a binomial is a permutation if and only if
\(\gcd(r,q-1)=1\), \((-a)^{\ell_e(q)}\ne1\), and
\(r\ell_h(q)\equiv1\pmod{\ell_e(q)}\) for some \(1\leqslant h<e\) coprime to
\(e\).  This proves a conjecture of Masuda, Rubio, and Santiago: every
permutation binomial of this form arises from
\((X^{q^h}+aX)\circ X^r\) for a suitable \(h\).  We also determine the exact number of distinct permutation functions represented by
this family.  As a further consequence, we completely classify
the broader family \(X^r(X^{d(q-1)}+a)\) in the coprime-index case
\(\gcd(d,\ell_e(q))=1\).  The new ingredient in the main classification is the
necessity argument: selected Hermite power sums are organized so that Lucas' theorem
turns their coefficients into digit conditions; a Farey-guided local argument
then forces successive base-\(q\) digits, and cyclic rotations yield the inverse
congruence.  In characteristic \(2\), a mod-\(4\) lift to an auxiliary ring
retains endpoint information lost modulo \(2\).
\end{abstract}
\maketitle

\section{Introduction}

A polynomial over a finite field is a \emph{permutation polynomial} if it
induces a bijection.  A basic problem in finite-field arithmetic is to identify
sparse families for which this property admits an exact arithmetic
characterization.  For a monomial \(X^n\) over \(\mathbb F_Q\), the answer is
simply \(\gcd(n,Q-1)=1\).  Already for binomials, however, the exponents and
coefficients interact nontrivially.

Let \(q\) be a prime power, let \(e\geqslant1\), and let \(r\geqslant1\) be an
integer.  We study the following family of binomials over \(\mathbb F_{q^e}\):
\begin{equation}
 f(X):=X^r(X^{q-1}+a)\in\mathbb F_{q^e}[X],
 \qquad a\in\mathbb F_{q^e}^{\times}.
 \label{eq:family}
\end{equation}
Since \(\zeta^{q-1}=1\) for every \(\zeta\in\mathbb F_q^{\times}\), we have
\(f(\zeta x)=\zeta^rf(x)\) for all \(\zeta\in\mathbb F_q^{\times}\),
\(x\in\mathbb F_{q^e}\).
Thus scaling the input by \(\zeta\in\mathbb F_q^{\times}\) scales the
output by \(\zeta^r\).  This base-field symmetry is a useful structural feature
of the family~\eqref{eq:family}.
When \(e=1\), \eqref{eq:family} gives \(f(x)=(1+a)x^r\) for
\(x\in\mathbb F_q^{\times}\).  Hence \(f\) permutes \(\mathbb F_q\) if and
only if \(a\ne-1\) and \(\gcd(r,q-1)=1\).  We therefore assume
\(e\geqslant2\) from now on.

For \(j\geqslant0\), write \(\ell_j(q):=(q^j-1)/(q-1)\).
Our main theorem gives a complete arithmetic classification in every extension
degree.

\begin{theorem}
\label{thm:main}
Let \(q\) be a prime power, let \(e\geqslant2\) and \(r\geqslant1\) be
integers, and let \(a\in\mathbb F_{q^e}^{\times}\).  Then \(X^{r}(X^{q-1}+a)\) permutes
\(\mathbb F_{q^e}\) if and only if
\begin{enumerate}[label=\textup{(\roman*)},ref=\roman*]
\item \label{item:coprime} \(\gcd(r,q-1)=1\);
\item \label{item:norm} \((-a)^{\ell_e(q)}\ne1\);
\item \label{item:inverse} there exists \(h\) with
\(1\leqslant h<e\), \(\gcd(h,e)=1\), and
\(r\ell_h(q)\equiv1\pmod{\ell_e(q)}\).
\end{enumerate}
\end{theorem}

Condition~\textup{(\ref{item:inverse})} is equivalent to the following:
there exists \(h\) with \(1\leqslant h<e\) and \(\gcd(h,e)=1\) such that
\[
 X^r(X^{q-1}+a)
 \equiv X^{rq^h}+aX^r
 = (X^{q^h}+aX)\circ X^r
 \pmod{X^{q^e}-X}.
\]
Under the three conditions of Theorem~\ref{thm:main},
Lemma~\ref{lem:elementary} shows that both factors in this composition
permute \(\mathbb F_{q^e}\).  Thus, in structural terms,
Theorem~\ref{thm:main} is a rigidity statement: despite the apparent freedom in the
exponent and coefficient, every permutation binomial in the family~\eqref{eq:family}
is, as a function on \(\mathbb F_{q^e}\), a composition
\((X^{q^h}+aX)\circ X^r\) of a permutation \(q\)-linearized binomial and a
permutation monomial, for some \(1\leqslant h<e\) with \(\gcd(h,e)=1\).  In
particular, Theorem~\ref{thm:main} proves the conjecture of
Masuda--Rubio--Santiago \cite[Conjecture~5.4]{MRS22}.

The criterion cleanly separates the coefficient and exponent arithmetic: the
coefficient \(a\) enters only through condition~\textup{(\ref{item:norm})},
whereas the possible residues of \(r\) modulo \(\ell_e(q)\) are determined
entirely by condition~\textup{(\ref{item:inverse})}.  The next two corollaries make this
separation explicit.

\begin{corollary}
\label{cor:number-residues}
For every prime power \(q\) and every \(e\geqslant2\),
condition~\textup{(\ref{item:inverse})} admits exactly \(\varphi(e)\)
residue classes of \(r\) modulo \(\ell_e(q)\).
\end{corollary}

\begin{corollary}
\label{cor:number-functions}
Let \(q\) be a prime power and \(e\geqslant1\).  The number of distinct
permutations of \(\mathbb F_{q^e}\) represented by binomials
\(X^r(X^{q-1}+a)\), with \(r\geqslant1\) and
\(a\in\mathbb F_{q^e}^{\times}\), is
\((q-2)\ell_e(q)\,\varphi(e)\,\varphi(q^e-1)/\varphi(\ell_e(q))\).
\end{corollary}

The proofs of Corollaries~\ref{cor:number-residues} and~\ref{cor:number-functions}
are given at the end of Section~\ref{sec:propagation}.

\subsection*{Related work and novelty}

Two strands of earlier work place Theorem~\ref{thm:main} in context.  In the
broader study of sparse permutation polynomials, Akbary--Ghioca--Wang \cite{AGW09} studied permutation polynomials with prescribed support,
while Masuda--Zieve \cite{MZ09} obtained general arithmetic
restrictions and existence results for permutation binomials.  For
the specific family~\eqref{eq:family}, Hou \cite{Hou16} obtained the \(e=2\)
characterization, and Liu \cite{Liu19} studied \(e=3\) in odd characteristic together with
partial results for general extension degree; see also the
summary in \cite[Proposition~1.1]{MRS22}.  Most directly,
Masuda--Rubio--Santiago \cite{MRS22}, using the Masuda--Panario--Wang form of
Hermite's criterion \cite[Corollary~2]{MPW06}, classified the
family~\eqref{eq:family} for arbitrary \(q\) in extension degrees \(e=2,3,4\)
and for odd prime \(q\) in degrees \(e=5,6\) \cite[Theorem~1.3]{MRS22}.  They also constructed
permutation binomials by composing a permutation monomial with
\(X^{q^h}+aX\) \cite[Theorem~3.4]{MRS22} and conjectured that these
constructions exhaust the family~\eqref{eq:family}
\cite[Conjecture~5.4]{MRS22}.  Theorem~\ref{thm:main} resolves this conjecture
completely for every prime power \(q\) and every extension degree \(e\geqslant2\).
Hou--Pallozzi Lavorante \cite{HPL23} subsequently developed an equivalence
framework for permutation binomials and obtained further nonexistence results
for the broader family \(X^r(X^{d(q-1)}+a)\), principally over
\(\mathbb F_{q^2}\).

For the nontrivial case \(q>2\), Masuda--Rubio--Santiago \cite{MRS22} also
proved several necessary restrictions valid for arbitrary \(e\).  In
particular, they showed that the least residue of \(r\) modulo \(\ell_e(q)\)
must be congruent to \(1\pmod q\) \cite[Proposition~4.1]{MRS22}, exactly the least-significant
base-\(q\) digit condition recovered below in Lemma~\ref{lem:pure}.  Their complete classifications, however, were limited to the fixed extension
degrees listed above.  Thus earlier work already supplies a uniform
least-significant-digit restriction for arbitrary \(e\) in the nontrivial case
\(q>2\).  What remained was to propagate this initial restriction far enough to
determine the entire residue class of the exponent.  The argument below achieves
this uniformly in the extension degree, with the final cyclic-rotation argument
completing the passage to the full residue class.  Section~\ref{sec:general-d}
then applies the resulting classification to the broader family
\(X^r(X^{d(q-1)}+a)\), completely settling the case
\(\gcd(d,\ell_e(q))=1\) and isolating the first genuinely new regime.

\subsection*{Proof strategy}

The sufficiency direction follows from the construction in
\cite[Theorem~3.4]{MRS22} and is recorded here in Lemma~\ref{lem:elementary}.
The new work lies in the necessity direction, where the key difficulty is local-to-global:
the known least-significant-digit restriction controls only the first base-\(q\) digit
of the exponent residue, whereas the classification requires its entire residue class
modulo \(\ell_e(q)\).  The necessity proof has three stages.

First, selected Hermite power sums, together with Lucas' theorem, convert the
permutation hypothesis into digit restrictions.  In particular, the case \(N=1\)
forces \(\rho\equiv1\pmod q\), which supplies the initial base-\(q\) digit and the
starting Farey cell.  Second, the Farey intervals organize the relevant floor data.
At level \(N\), the current Farey cell is first localized to an interval on which the
floor data are stable, and the local coefficient argument then forces the next
digit \(v_N\).  Passing from \(F_N\) to \(F_{N+1}\) gives the next cell, so the same
argument can be iterated.  This produces a binary quotient recursion uniformly in the
extension degree.  Third, that recursion is read as a length-\(e\) binary word in
base \(q\).  Ordering its cyclic rotations converts the digit information into the
inverse congruence in condition~\textup{(\ref{item:inverse})}.

Only the second stage depends on the characteristic.  In odd characteristic the
next-digit obstruction is detected modulo \(p\).  In characteristic \(2\), reduction
modulo \(2\) loses the endpoint distinction needed by the forcing argument, so we
lift the relevant coefficient identities to an auxiliary ring of characteristic
\(4\), retaining the missing mod-\(4\) information.  After the next digit is forced,
the passage to the next Farey level and the final cyclic-rotation argument are
characteristic-free.  Related finite-word constructions and lexicographic orderings
of cyclic rotations are discussed in \cite{deLuca97,Aigner2013,Zamboni26}; the
consequences needed below are proved here.  Section~\ref{sec:coefficients} carries
out the first stage; Sections~\ref{sec:endpoints} and~\ref{sec:selector} develop the
interval machinery and the one-step digit forcing; and Section~\ref{sec:propagation}
iterates that forcing and completes the third stage.  Finally,
Section~\ref{sec:general-d} records the coprime-index extension to the broader
family \(X^r(X^{d(q-1)}+a)\) and the remaining open regime.

\section{Hermite power sums and digit conditions}
\label{sec:coefficients}

This section converts the permutation hypothesis into the coefficient and digit
constraints used in the Farey recursion.  We first isolate the elementary
permutation conditions and the relevant power-sum identity, and then use Lucas'
theorem to extract the digit information that initiates the recursion.

\begin{lemma}
\label{lem:elementary}
Let \(q\) be a prime power, let \(e\geqslant2\) and \(r\geqslant1\) be
integers, and let \(a\in\mathbb F_{q^e}^{\times}\).  If
\(X^r(X^{q-1}+a)\) permutes \(\mathbb F_{q^e}\), then
\(\gcd(r,q-1)=1\) and \((-a)^{\ell_e(q)}\ne1\).  Conversely, these
two conditions together with
\(r\ell_h(q)\equiv1\pmod{\ell_e(q)}\) for some
\(1\leqslant h<e\) coprime to \(e\) imply that \(X^r(X^{q-1}+a)\) permutes
\(\mathbb F_{q^e}\).
\end{lemma}

\begin{proof}
Write \(f(X)=X^r(X^{q-1}+a)\).  Fix \(x\in\mathbb F_{q^e}^{\times}\).
Since \(f(0)=0\) and \(f\) is
injective, \(f(x)\ne0\).  For \(\zeta\in\mathbb F_q^{\times}\),
\(f(\zeta x)=\zeta^{r}f(x)\).  Injectivity of \(f\)
therefore forces \(\zeta\mapsto\zeta^r\) to be injective on
\(\mathbb F_q^{\times}\); hence \(\gcd(r,q-1)=1\).  The equation \(x^{q-1}=-a\) has a nonzero solution exactly when
\((-a)^{\ell_e(q)}=1\).  This proves condition~\textup{(\ref{item:norm})}.

For the converse, the inverse congruence implies
\(\gcd(r,\ell_e(q))=1\).  Together with \(\gcd(r,q-1)=1\), this gives
\(\gcd(r,q^e-1)=1\), so \(X^{r}\) is a permutation monomial.
Since \(\gcd(h,e)=1\),
\(\gcd(q^h-1,q^e-1)=q-1\).  Hence the \(\mathbb F_q\)-linear map induced by
\(X^{q^h}+aX\) has a nonzero kernel exactly when
\((-a)^{\ell_e(q)}=1\).  Finally
\[
 r(q^h-1)\equiv q-1\pmod{q^e-1},
\]
so \(f\) induces the composition of these two permutations.
\end{proof}

\smallskip
\noindent\textbf{Hermite's criterion.}\ For a finite field \(K=\mathbb F_Q\) of characteristic \(p\), a polynomial
\(G\in K[X]\) permutes \(K\) if and only if \textup{(i)} it has exactly
one zero in \(K\), and \textup{(ii)} for every
\(1\leqslant n\leqslant Q-2\) with \(p\nmid n\), the reduction of
\(G^n\) modulo \(X^Q-X\) has degree at most \(Q-2\)
\cite{Hermite1863,Dickson1896}; see also
\cite[Theorem~7.4]{LidlNiederreiter1997}.  In particular, if
\(G\) permutes \(K\), permutation invariance of the sum gives
\[
 \sum_{x\in K}G(x)^n=0
 \qquad\text{for all }1\leqslant n\leqslant Q-2.
\]
This includes the case \(p\mid n\).

For an integer \(z\) and a positive integer \(M\), write \([z]_M\) for the least
nonnegative residue modulo \(M\).  Put
\(\rho:=[r]_{\ell_e(q)}\), \(0\leqslant\rho<\ell_e(q)\), reserving \(r\) for the actual exponent.  Write
\(q=p^m\), with \(p\) prime.  For \(1\leqslant\tau<\ell_e(q)\), put
\[
 H_\tau:=q^e-1-(q-1)\tau,\qquad y_\tau:=[\rho\tau]_{\ell_e(q)},
\]
and define
\[
 P_{\rho,\tau}(X):=\sum_{i=0}^{q-2}
 \binom{H_\tau}{y_\tau+i\ell_e(q)}X^i\in\mathbb F_p[X],
\]
where an out-of-range binomial coefficient is zero.  The quantities \(H_\tau\), \(y_\tau\), and
\(P_{\rho,\tau}\) are defined relative to the fixed choices of \(q\) and
\(e\); this dependence is suppressed from the notation.  The quantity
\(y_\tau\) also depends on the current residue \(\rho\).

The identity below is \cite[Corollary~2]{MPW06} specialized to \(N=H_\tau\),
with the unrestricted formulation given in \cite[Proposition~2.1]{MRS22}.  We
include a direct derivation to avoid the degree hypothesis in the original statement.

\begin{lemma}[Power-sum identity]
\label{lem:power-sum}
For \(1\leqslant\tau<\ell_e(q)\) and \(\rho=[r]_{\ell_e(q)}\),
\begin{equation}
 \sum_{x\in\mathbb F_{q^e}}\bigl(x^r(x^{q-1}+a)\bigr)^{H_\tau}
 =-a^{H_\tau-y_\tau}P_{\rho,\tau}(a^{-\ell_e(q)}).
 \label{eq:power-sum}
\end{equation}
If \(X^r(X^{q-1}+a)\) permutes \(\mathbb F_{q^e}\), then
\(P_{\rho,\tau}(a^{-\ell_e(q)})=0\); in particular,
\(P_{\rho,\tau}\) cannot be a nonzero monomial.
\end{lemma}

\begin{proof}
Expand \(f(x)^{H_\tau}\).  For a positive integer \(d\),
\(\sum_xx^d\) is \(-1\) when \((q^e-1)\mid d\) and zero otherwise, and
\[
 rH_\tau+(q-1)A
 =r(q^e-1)+(q-1)(A-r\tau),
\]
so, since \(q^e-1=(q-1)\ell_e(q)\),
\[
 (q^e-1)\mid rH_\tau+(q-1)A
 \Longleftrightarrow
 A\equiv r\tau\equiv\rho\tau\pmod{\ell_e(q)}.
\]
Since \(H_\tau<q^e-1=(q-1)\ell_e(q)\), the relevant lower indices are
exactly \(y_\tau+i\ell_e(q)\), \(0\leqslant i\leqslant q-2\), yielding
\eqref{eq:power-sum}.  If \(f\) permutes \(\mathbb F_{q^e}\),
the power-sum consequence of Hermite's criterion makes the left-hand side zero.  Since
\(a^{-\ell_e(q)}\ne0\), a one-term polynomial cannot vanish there.
\end{proof}

To determine which coefficients of \(P_{\rho,\tau}\) survive modulo \(p\),
we use Lucas' theorem in its digitwise form.

\smallskip
\noindent\textbf{Lucas' theorem.} If \(p\) is prime and
\(n=\displaystyle\sum_{i\geqslant0}n_ip^i\), \(k=\displaystyle\sum_{i\geqslant0}k_ip^i\), with
\(0\leqslant n_i,k_i<p\), then \(\displaystyle \binom nk\equiv\prod_{i\geqslant0}\binom{n_i}{k_i}\pmod p\), and
\[
 \binom nk\not\equiv0\pmod p\Longleftrightarrow k_i\leqslant n_i
 \text{ for all }i.
\]
For a modern reference, see \cite[Theorem~10.2.1]{AndreescuAndrica2009};
for the original paper, see \cite{Lucas1878}.

We first apply Lucas' theorem at \(\tau=\ell_N(q)\) for \(1\leqslant N<e\).

\begin{lemma}[Power sums at \(\tau=\ell_N(q)\)]
\label{lem:pure}
Let \(1\leqslant N<e\), set
\(n=\lfloor \ell_N(q)\rho/\ell_e(q)\rfloor\) and
\(\delta=[n-\rho]_{q^N}\).  Then
\(P_{\rho,\ell_N(q)}\) has at most one nonzero term: if
\(\delta\leqslant q-2\), it is a nonzero scalar multiple of
\(X^\delta\), while if \(\delta\geqslant q-1\), it is zero.
In particular, if \(X^r(X^{q-1}+a)\) permutes \(\mathbb F_{q^e}\), then the case
\(N=1\) forces
\[
 1\leqslant \rho<\ell_e(q),\qquad \rho\equiv1\pmod q.
\]
\end{lemma}

\begin{proof}
Here \(H_{\ell_N(q)}=q^e-q^N\).  In base \(p\), its lowest \(mN\)
digits are zero and all higher digits are \(p-1\).  By
Lucas' theorem, a lower index survives exactly when it is
divisible by \(q^N\).  Since
\(\rho\ell_N(q)=n\ell_e(q)+y_{\ell_N(q)}\) and \(\ell_e(q)\equiv \ell_N(q)\pmod{q^N}\),
\[
 y_{\ell_N(q)}+i\ell_e(q)\equiv(\rho-n+i)\ell_N(q)\pmod{q^N}.
\]
As \(\ell_N(q)\) is invertible modulo \(q^N\), there is at most one
\(i\in\{0,\ldots,q-2\}\) for which the coefficient is nonzero, namely
\([n-\rho]_{q^N}\).  This proves the assertion about
\(P_{\rho,\ell_N(q)}\).  At \(N=1\) one has \(n=0\).  If \(f\) permutes \(\mathbb F_{q^e}\),
Lemma~\ref{lem:power-sum} rules out the nonzero monomial case; hence \([-\rho]_q=q-1\), so \(\rho\equiv1\pmod q\).
The case \(\rho=0\) would instead give the forbidden nonzero constant
polynomial, and is therefore excluded.
\end{proof}

The case \(N=1\) of Lemma~\ref{lem:pure} recovers
\cite[Proposition~4.1]{MRS22}.
At the first level, the lemma initializes the base-\(q\) digit recursion; for general \(N\), the same one-term obstruction will be used below to remove the boundary part of the current Farey state.

\section{Farey intervals and binary sequences}
\label{sec:endpoints}

Fix an integer \(q>2\) in this section and write
\(\ell_j(q)=(q^j-1)/(q-1)\).  The Farey sequences and mechanical words below are independent
of \(q\), whereas the numerical quantities obtained by reading those words in base
\(q\) depend on the chosen base.  For recurring base-\(q\) quantities, we display this dependence by a superscript
\((q)\); proof-local auxiliary quantities remain unadorned.

Let \(F_N\) be the increasing Farey sequence of reduced fractions in \([0,1]\)
with denominator at most \(N\), including \(0=0/1\) and \(1=1/1\).  For
\(N\geqslant4\), its beginning and end are
\[
 F_N=\left(\frac01,\frac1N,\frac1{N-1},\ldots,
 \frac{N-2}{N-1},\frac{N-1}{N},\frac11\right).
\]
Only the formulas below and standard facts about Farey neighbors will be used.
Our goal is to turn each Farey interval into stable floor data that can be propagated
one level at a time; all consequences needed later are proved explicitly.

For a \emph{word} \(w=(w_i)_{i\geqslant0}\in\{0,1\}^{\mathbb N}\), put
\[
\mathcal{V}_N^{(q)}(w):=\sum_{i=0}^{N-1}w_iq^{N-1-i}\quad\text{for all }N\geqslant1,\qquad
\mathcal{V}_0^{(q)}(w):=0.
\]
Recall that for \(t\in\mathbb R\), \(\lfloor t\rfloor\) is the greatest integer at most \(t\), while \(\lceil t\rceil\) is the least integer at least \(t\).
For \(0\leqslant\theta\leqslant1\), define the words
\(c^+(\theta)=(c_i^+(\theta))_{i\geqslant0}\) and
\(c^-(\theta)=(c_i^-(\theta))_{i\geqslant0}\) by
\[
 c_i^+(\theta):=
 \lfloor(i+2)\theta\rfloor-\lfloor(i+1)\theta\rfloor,\qquad
 c_i^-(\theta):=
 \lceil(i+2)\theta\rceil-\lceil(i+1)\theta\rceil.
\]
These are, respectively, the \emph{lower mechanical word} and the
\emph{upper mechanical word}; see \cite[Definition~8.12]{Aigner2013}.

Suppose \(\theta\in(0,1)\) is rational, and write \(\theta=c/d\) in lowest terms.
Then both sequences are \(d\)-periodic; see \cite[Example~8.18]{Aigner2013}.
Put
\begin{align}
 B_\theta^{(q)}&:=\mathcal{V}_d^{(q)}(c^-(\theta)),&
 \mu_\theta^{(q)}&:=\frac{B_\theta^{(q)}+q-1}{\ell_d(q)+q-1}
 \label{eq:endpoints}\\
 \alpha_\theta^{(q)}&:=\frac{B_\theta^{(q)}}{\ell_d(q)},&
 \beta_\theta^{(q)}&:=\frac{B_\theta^{(q)}+q-1}{\ell_d(q)}.\notag
\end{align}
For \(0\leqslant i<d\), the floor and ceiling differences can differ only if one
of \((i+1)\theta\) and \((i+2)\theta\) is an integer.  Since \(\theta=c/d\)
is in lowest terms, this occurs only for \(i=d-2,d-1\), where the integer is
\(d\theta=c\).  Hence, within each period of length \(d\), the words
\(c^-(\theta)\) and \(c^+(\theta)\) differ only in the last two positions:
\(c^-(\theta)\) ends in \(01\), whereas \(c^+(\theta)\) ends in \(10\).
Consequently,
\[
 \mathcal{V}_d^{(q)}(c^+(\theta))=B_\theta^{(q)}+q-1,\qquad
 \sum_{i\geqslant0}c_i^\pm(\theta) q^{-i-1}
 =\begin{cases}
   \beta_\theta^{(q)}/(q-1),&\text{for }c^+(\theta),\\
   \alpha_\theta^{(q)}/(q-1),&\text{for }c^-(\theta).
  \end{cases}
\]
Each period contains both digits.
It follows that \(0<B_\theta^{(q)}<B_\theta^{(q)}+q-1<\ell_d(q)\), and hence
\(0<\alpha_\theta^{(q)}<\mu_\theta^{(q)}<\beta_\theta^{(q)}<1\).
To allow the interval notation below to include the boundary pairs with
\(\theta_-=0\) or \(\theta_+=1\), set \(\beta_0^{(q)}:=0\) and
\(\alpha_1^{(q)}:=1\); we also use the all-zero sequence \(c^+(0)\) and the
all-one sequence \(c^-(1)\).  The next lemma translates these periodic words
into the floor and ceiling data used below.

\begin{lemma}
For every rational \(\theta\in(0,1)\) and every
\(N\geqslant1\),
\begin{align}
 \lfloor \ell_N(q)\beta_\theta^{(q)}\rfloor&=q\mathcal{V}_{N-1}^{(q)}(c^+(\theta)),&
 \lceil \ell_N(q)\alpha_\theta^{(q)}\rceil-1&=q\mathcal{V}_{N-1}^{(q)}(c^-(\theta))
 \label{eq:four-endpoints}\\
 \lfloor(\ell_N(q)+1)\beta_\theta^{(q)}\rfloor&=\mathcal{V}_N^{(q)}(c^+(\theta)),&
 \lceil(\ell_N(q)+1)\alpha_\theta^{(q)}\rceil-1&=\mathcal{V}_N^{(q)}(c^-(\theta)).\notag
\end{align}
The formulas involving \(\beta_0^{(q)}\) and \(\alpha_1^{(q)}\) hold as well.
\end{lemma}

\begin{proof}
Write \(\theta=c/d\) in lowest terms and, for \(0\leqslant z<d\), set
\[
 d_i(z)=\left\lfloor\frac{z+(i+1)c}{d}\right\rfloor-
 \left\lfloor\frac{z+ic}{d}\right\rfloor,
 \qquad T(z)=\sum_{i\geqslant0}d_i(z)q^{-i-1}.
\]
The first \(j\) digits sum to \(\lfloor(z+jc)/d\rfloor\), a nondecreasing
function of \(z\).  If \(z_1<z_2\), choose \(j\) with \(jc\equiv d-z_2\pmod d\); then the two
partial sums differ, and at their first differing digit the smaller
parameter gives \(0\) and the larger \(1\).  Since \(q>2\), later digits
cannot reverse this inequality, so \(T(z)\) is strictly increasing.

The words \(c^+(\theta)\), \(c^-(\theta)\) correspond to initial parameters \(c\), \(c-1\).
If \(t\) is the associated base-\(q\) fraction, \(t_N\) its tail after \(N\)
digits, and \(D\) the \(N\)-th digit, then
\(q^Nt=\mathcal{V}_N^{(q)}(c^\pm(\theta))+t_N\) and \(\mathcal{V}_N^{(q)}(c^\pm(\theta))=q\mathcal{V}_{N-1}^{(q)}(c^\pm(\theta))+D\).  For \(c^+(\theta)\), the new
parameter is \([(N+1)c]_d\); hence \(D=1\) iff it is below \(c\), so
monotonicity of \(T\) gives \(D=1\Rightarrow t_N<t\) and
\(D=0\Rightarrow t_N\geqslant t\).  Therefore
\(\lfloor(q^N-1)t\rfloor=q\mathcal{V}_{N-1}^{(q)}(c^+(\theta))\).  For \(c^-(\theta)\), the new parameter is
\([(N+1)c-1]_d\); similarly \(D=1\Rightarrow t_N\leqslant t\) and
\(D=0\Rightarrow t_N>t\), giving
\(\lceil(q^N-1)t\rceil-1=q\mathcal{V}_{N-1}^{(q)}(c^-(\theta))\).  The weak inequalities include the
case \(d\mid N\).

Finally,
\[
 (\ell_N(q)+1)(q-1)t=\mathcal{V}_N^{(q)}(c^\pm(\theta))+t_N+(q-2)t.
\]
Every shifted period contains both digits, so
\(0<t,t_N<1/(q-1)\) and hence \(0<t_N+(q-2)t<1\).  Taking the floor, or
the ceiling minus one, gives the last two identities in
\eqref{eq:four-endpoints}; the outer endpoints follow from
\(\ell_N(q)-1=q\ell_{N-1}(q)\).
\end{proof}

We next record the Farey-neighbor facts that control the passage from \(F_N\) to \(F_{N+1}\).  For
fractions \(c_1/d_1<c_2/d_2\), write their mediant as
\((c_1+c_2)/(d_1+d_2)\).

\begin{lemma}
\label{lem:farey}
Suppose \(c_1/d_1<c_2/d_2\) are adjacent in \(F_N\), with
\(\gcd(c_1,d_1)=\gcd(c_2,d_2)=1\).  Then
\[
 d_1c_2-c_1d_2=1,
 \qquad D=d_1+d_2>N.
\]
For every integer \(k\) with \(1\leqslant k<D\),
\begin{equation}
 \left\lfloor\frac{kc_1}{d_1}\right\rfloor
 =\left\lceil\frac{kc_2}{d_2}\right\rceil-1.
 \label{eq:farey-floor}
\end{equation}
At \(k=D\), the corresponding values are
\[
 \left\lfloor\frac{Dc_1}{d_1}\right\rfloor=c_1+c_2-1,
 \qquad
 \left\lceil\frac{Dc_2}{d_2}\right\rceil=c_1+c_2+1.
\]
The sequence \(F_{N+1}\) is obtained from \(F_N\) by inserting
the mediant precisely for adjacent pairs with \(d_1+d_2=N+1\).
\end{lemma}

\begin{proof}
For the standard Farey sequences \(F_N\) used here, consecutive fractions satisfy
\(d_1c_2-c_1d_2=1\); every intermediate reduced fraction has denominator at
least \(D=d_1+d_2\), while the mediant has denominator exactly \(D\).
Moreover, \(F_{N+1}\) inserts that mediant precisely when \(D=N+1\); see
\cite[Theorems~5.3--5.5]{Apostol1990}.  In particular, \(D>N\).
If \(1\leqslant k<D\), the open interval
\((kc_1/d_1,kc_2/d_2)\) contains no integer: if an integer \(t\) lay
in this interval, then after reduction \(t/k\) would give an intermediate
fraction of denominator less than \(D\).  This proves
\eqref{eq:farey-floor}.  Finally,
\[
 \frac{Dc_1}{d_1}=c_1+c_2-\frac1{d_1},\qquad
 \frac{Dc_2}{d_2}=c_1+c_2+\frac1{d_2},
\]
which gives the endpoint assertions.
\end{proof}

Combining the endpoint identities with Farey adjacency gives the intervals used
in the induction below.  For an adjacent pair \(\theta_-<\theta_+\) in \(F_N\), let
\[
 K_N^{(q)}(\theta_-,\theta_+):=[\beta_{\theta_-}^{(q)},\alpha_{\theta_+}^{(q)}).
\]
\begin{lemma}
\label{lem:intervals}
Let \(\theta_-=c_1/d_1<\theta_+=c_2/d_2\) be adjacent in \(F_N\), with
\(\gcd(c_1,d_1)=\gcd(c_2,d_2)=1\), and put \(D=d_1+d_2\).  The interval \(K_N^{(q)}(\theta_-,\theta_+)\) is nonempty.  On it,
\(\lfloor\ell_N(q)y\rfloor\) is constant, say equal to \(n\).
If \(D>N+1\), then \(\lfloor(\ell_N(q)+1)y\rfloor\) is also constant.
If \(D=N+1\), it takes exactly the two values \(n,n+1\), with its
unique change at \(\mu_{\theta_0}^{(q)}\), where \(\theta_0=(c_1+c_2)/(d_1+d_2)\).
In this case,
\[
 B_{\theta_0}^{(q)}=qn+1,
 \qquad
 \mu_{\theta_0}^{(q)}=\frac{n+1}{\ell_N(q)+1},
 \qquad
 \beta_{\theta_-}^{(q)}<\mu_{\theta_0}^{(q)}<\alpha_{\theta_+}^{(q)}.
\]
\end{lemma}

\begin{proof}
By \eqref{eq:farey-floor}, \(c^+(\theta_-)\) and \(c^-(\theta_+)\) agree
in their first \(D-2\) digits, while the next digits are \(0\) and \(1\) by
Lemma~\ref{lem:farey}.  A first differing binary digit dominates the tail because \(q>2\), so
\(\beta_{\theta_-}^{(q)}<\alpha_{\theta_+}^{(q)}\).  Since \(D>N\), the first
\(N-1\) digits agree, and \eqref{eq:four-endpoints} gives
\[
 \lfloor\ell_N(q)\beta_{\theta_-}^{(q)}\rfloor
 =\lceil\ell_N(q)\alpha_{\theta_+}^{(q)}\rceil-1
 =q\mathcal{V}_{N-1}^{(q)}(c^+(\theta_-)),
\]
so \(\lfloor\ell_N(q)y\rfloor\) is constant on \(K_N^{(q)}\).

If \(D>N+1\), the first \(N\) digits also agree, so
\(\lfloor(\ell_N(q)+1)y\rfloor\) is constant.  If \(D=N+1\), its endpoint
values are \(n,n+1\).  For the mediant \(\theta_0\), the first
\(N-1\) digits of \(c^-(\theta_0)\) are the common ones and its last two
period digits are \(01\), hence \(B_{\theta_0}^{(q)}=qn+1\).  From
\eqref{eq:endpoints} and \(\ell_{N+1}(q)+q-1=q(\ell_N(q)+1)\),
\(\mu_{\theta_0}^{(q)}=(n+1)/(\ell_N(q)+1)\); the endpoint values place this point
strictly inside \(K_N^{(q)}\) and exclude any third quotient value.
\end{proof}

For the induction below, we define a family of left-closed, right-open intervals
\(J_N^{(q)}(\theta_-,\theta_+)\).  Set \(J_1^{(q)}(0,1):=[0,1)\).  For
\(N\geqslant2\) and adjacent \(\theta_-<\theta_+\) in \(F_N\), let
\(J_N^{(q)}(\theta_-,\theta_+)\) be the left-closed, right-open interval whose
left endpoint is \(0\) if \(\theta_-=0\), \(\mu_{\theta_-}^{(q)}\) if
\(\theta_-\in(0,1)\) has reduced denominator \(N\), and
\(\beta_{\theta_-}^{(q)}\) otherwise, and whose right endpoint is \(1\) if
\(\theta_+=1\), \(\mu_{\theta_+}^{(q)}\) if \(\theta_+\in(0,1)\) has reduced
denominator \(N\), and \(\alpha_{\theta_+}^{(q)}\) otherwise.  By
Lemma~\ref{lem:intervals},
\[
 K_N^{(q)}(\theta_-,\theta_+)
 =[\beta_{\theta_-}^{(q)},\alpha_{\theta_+}^{(q)})
 \subseteq J_N^{(q)}(\theta_-,\theta_+).
\]

\begin{corollary}[One-step interval decomposition]
\label{cor:one-step-intervals}
Let \(\theta_-=c_1/d_1<\theta_+=c_2/d_2\) be adjacent in \(F_N\), with
\(\gcd(c_1,d_1)=\gcd(c_2,d_2)=1\), and set
\(\theta_0=(c_1+c_2)/(d_1+d_2)\).  Then
\((F_{N+1}\setminus F_N)\cap(\theta_-,\theta_+)\) contains at most one point.
If \(d_1+d_2>N+1\), this set is empty and
\[
 K_N^{(q)}(\theta_-,\theta_+)=J_{N+1}^{(q)}(\theta_-,\theta_+).
\]
If \(d_1+d_2=N+1\), then
\((F_{N+1}\setminus F_N)\cap(\theta_-,\theta_+)=\{\theta_0\}\), and
\[
 K_N^{(q)}(\theta_-,\theta_+)
 =J_{N+1}^{(q)}(\theta_-,\theta_0)
 \sqcup J_{N+1}^{(q)}(\theta_0,\theta_+).
\]
\end{corollary}

\begin{proof}
By Lemma~\ref{lem:farey}, one has \(d_1+d_2>N\), and a new point of
\(F_{N+1}\) can lie between \(\theta_-\) and \(\theta_+\) only when
\(d_1+d_2=N+1\); in that case it is their mediant \(\theta_0\).
If \(d_1+d_2>N+1\), the two fractions remain adjacent in \(F_{N+1}\), and the
definition of \(J_{N+1}^{(q)}\) gives the first equality.  If
\(d_1+d_2=N+1\), Lemma~\ref{lem:intervals} gives
\(\beta_{\theta_-}^{(q)}<\mu_{\theta_0}^{(q)}<\alpha_{\theta_+}^{(q)}\), while
the definition of \(J_{N+1}^{(q)}\) gives
\[
 J_{N+1}^{(q)}(\theta_-,\theta_0)
 =[\beta_{\theta_-}^{(q)},\mu_{\theta_0}^{(q)}),\qquad
 J_{N+1}^{(q)}(\theta_0,\theta_+)
 =[\mu_{\theta_0}^{(q)},\alpha_{\theta_+}^{(q)}).
\]
Their disjoint union is \(K_N^{(q)}(\theta_-,\theta_+)\).
\end{proof}

\section{Forcing the next digit}
\label{sec:selector}

The first stage of the necessity proof has supplied the initial congruence
\(\rho\equiv1\pmod q\).  We now carry out the one-step argument in the second
stage: at a fixed Farey level \(N\), we localize \(\rho/\ell_e(q)\) to an interval on
which the floor data are stable and then force the next base-\(q\) digit.  We work in
the finite-field setting of Section~\ref{sec:coefficients}, with \(q=p^m>2\), and
combine the Farey data from Section~\ref{sec:endpoints} with the coefficient
conditions.  We retain the superscript \((q)\) on the base-\(q\) quantities introduced
there; the Farey and mechanical-word data remain unadorned.  For \(N\geqslant1\)
and \(0\leqslant\xi<1\), let
\(\theta_N^-(\xi)<\theta_N^+(\xi)\) be the unique adjacent terms of
\(F_N\) satisfying \(\theta_N^-(\xi)\leqslant\xi<\theta_N^+(\xi)\).
Define
\begin{equation}
\begin{aligned}
 \varepsilon_N(\xi)&:=\lfloor N\xi\rfloor-\lfloor(N-1)\xi\rfloor,\\
 \sigma_N^{(q)}(\xi)&:=1+\sum_{k=2}^{N-1}
  \bigl(\lfloor k\xi\rfloor-\lfloor(k-1)\xi\rfloor\bigr)q^k
  =1+\sum_{k=2}^{N-1}\varepsilon_k(\xi)q^k,\\
 n_N^{(q)}(\xi)&:=q\mathcal{V}_{N-1}^{(q)}(c^+(\xi)).
\end{aligned}
\label{eq:labels}
\end{equation}
For \(N=1\) the sum is empty, \(\mathcal{V}_0^{(q)}(c^+(\xi))=0\), and \(\lfloor\xi\rfloor=0\), so
\(\sigma_1^{(q)}(\xi)=1\) and \(n_1^{(q)}(\xi)=\varepsilon_1(\xi)=0\).

\begin{lemma}
\label{lem:labels}
For each \(N\geqslant1\), the functions
\(\varepsilon_N(\cdot)\), \(\sigma_N^{(q)}(\cdot)\), and \(n_N^{(q)}(\cdot)\) are constant
on every half-open interval \([u,v)\) with adjacent \(u<v\) in \(F_N\).  They satisfy
\[
 \begin{gathered}
 \varepsilon_N(\xi)\in\{0,1\},\quad
 1\leqslant\sigma_N^{(q)}(\xi)\leqslant\ell_N(q),\quad
 0\leqslant n_N^{(q)}(\xi)\leqslant\ell_N(q)-1,\\
 \sigma_N^{(q)}(\xi)\equiv1\pmod q,\quad n_N^{(q)}(\xi)\equiv0\pmod q,\\
 n_N^{(q)}(\xi)-\sigma_N^{(q)}(\xi)\equiv\varepsilon_N(\xi)-1\pmod{q-1}.
 \end{gathered}
\]
For
\(y\in K_N^{(q)}(\theta_N^-(\xi),\theta_N^+(\xi))\),
\[
 \lfloor\ell_N(q)y\rfloor=n_N^{(q)}(\xi),\qquad
 \lfloor(\ell_N(q)+1)y\rfloor\in\{n_N^{(q)}(\xi),n_N^{(q)}(\xi)+1\}.
\]
Moreover, for \(N\geqslant1\),
\(\sigma_{N+1}^{(q)}(\xi)=\sigma_N^{(q)}(\xi)+\varepsilon_N(\xi)q^N\), and, for
\(N\geqslant2\),
\begin{align}
 n_N^{(q)}(\xi)-qn_{N-1}^{(q)}(\xi)&=q\varepsilon_N(\xi),\notag\\
 \sigma_N^{(q)}(\xi)&\equiv1\pmod{q^2},\qquad
 n_N^{(q)}(\xi)\equiv q\varepsilon_N(\xi)\pmod{q^2}.
 \label{eq:strong-state}
\end{align}
\end{lemma}

\begin{proof}
Put \(u=\theta_N^-(\xi)\) and \(v=\theta_N^+(\xi)\).  For
\(k\leqslant N\), the right-continuous function
\(t\mapsto\lfloor kt\rfloor\) can change only at fractions of denominator
at most \(N\), hence is constant on \([u,v)\).  Therefore all floor
differences entering \eqref{eq:labels}, including the first \(N-1\) digits of
\(c^+(\xi)\), are constant there, proving the first assertion.  The bounds and
congruences follow from
\eqref{eq:labels} and \(q\ell_{N-1}(q)=\ell_N(q)-1\); modulo \(q-1\),
telescoping gives
\[
 n_N^{(q)}(\xi)\equiv\lfloor N\xi\rfloor,\qquad
 \sigma_N^{(q)}(\xi)\equiv1+\lfloor(N-1)\xi\rfloor.
\]
The quotient assertions follow from Lemma~\ref{lem:intervals}, since
constancy gives \(n_N^{(q)}(\xi)=n_N^{(q)}(u)\), while
\(n_N^{(q)}(u)=q\mathcal{V}_{N-1}^{(q)}(c^+(u))\).  The recurrence for \(\sigma_{N+1}^{(q)}(\xi)\) follows directly from
\eqref{eq:labels} (also for \(N=1\), since \(\varepsilon_1(\xi)=0\)).
For \(N\geqslant2\),
\[
 \mathcal{V}_{N-1}^{(q)}(c^+(\xi))=q\mathcal{V}_{N-2}^{(q)}(c^+(\xi))+\varepsilon_N(\xi),
\]
which gives the recurrence for \(n_N^{(q)}(\xi)\).
Finally, for \(N\geqslant2\), the sum defining \(\sigma_N^{(q)}(\xi)\) begins
at \(q^2\), while the recurrence for \(n_N^{(q)}(\xi)\), together with
\(n_{N-1}^{(q)}(\xi)\equiv0\pmod q\), gives \eqref{eq:strong-state}.
\end{proof}

\begin{lemma}[Farey-state localization]
\label{lem:farey-localization}
Let \(1\leqslant N<e\), let \(\theta_-<\theta_+\) be adjacent in \(F_N\), and suppose that
\(X^r(X^{q-1}+a)\) permutes \(\mathbb F_{q^e}\), with
\(\rho=[r]_{\ell_e(q)}\).  If
\begin{equation}
 \frac{\rho}{\ell_e(q)}\in J_N^{(q)}(\theta_-,\theta_+),\qquad
 \rho\equiv\sigma_N^{(q)}(\theta_-)\pmod{q^N},
 \label{eq:farey-condition}
\end{equation}
then
\[
 \frac{\rho}{\ell_e(q)}\in K_N^{(q)}(\theta_-,\theta_+),\qquad
 \left\lfloor\frac{\ell_N(q)\rho}{\ell_e(q)}\right\rfloor
 =n_N^{(q)}(\theta_-).
\]
\end{lemma}

\begin{proof}
Put \(u=\left\lfloor\ell_N(q)\rho/\ell_e(q)\right\rfloor\).
Lemma~\ref{lem:pure} gives \(1\leqslant\rho<\ell_e(q)\).  Since
\(-\ell_N(q)\leqslant u-\sigma_N^{(q)}(\theta_-)\leqslant\ell_N(q)-2\) and
\(q^N-\ell_N(q)=(q-2)\ell_N(q)+1\geqslant q-1\), the congruence condition in
Lemma~\ref{lem:pure} is equivalent here to
\(\sigma_N^{(q)}(\theta_-)\leqslant u\leqslant\sigma_N^{(q)}(\theta_-)+q-2\).
Thus the portion excluded by Lemma~\ref{lem:pure} is
\[
 \left[
 \frac{\sigma_N^{(q)}(\theta_-)}{\ell_N(q)},
 \frac{\sigma_N^{(q)}(\theta_-)+q-1}{\ell_N(q)}
 \right)\cap J_N^{(q)}(\theta_-,\theta_+).
\]
If an endpoint \(\vartheta\in(0,1)\) has reduced denominator \(N\), then
\(k\vartheta\notin\mathbb Z\) for \(1\leqslant k<N\).  Thus the floor
differences defining \(\sigma_N^{(q)}\) do not jump at \(\vartheta\), and
\(\sigma_N^{(q)}(\theta_-)=B_\vartheta^{(q)}\).  Indeed, the \(q^0\)- and \(q^1\)-digits
of \(\sigma_N^{(q)}(\theta_-)\) are \(1\) and \(0\), while for \(2\leqslant k<N\) the identity
\[
 \lceil(N+1-k)\vartheta\rceil-\lceil(N-k)\vartheta\rceil
 =\lfloor k\vartheta\rfloor-\lfloor(k-1)\vartheta\rfloor
\]
follows from \(N\vartheta\in\mathbb Z\).  Hence the interval excluded by
Lemma~\ref{lem:pure} is \([\alpha_\vartheta^{(q)},\beta_\vartheta^{(q)})\).  If
\(\vartheta\) is the right endpoint, its intersection with \(J_N^{(q)}\) is
\([\alpha_\vartheta^{(q)},\mu_\vartheta^{(q)})\); if it is the left endpoint, the
intersection is \([\mu_\vartheta^{(q)},\beta_\vartheta^{(q)})\).  Thus the
power-sum condition removes exactly the part of
\(J_N^{(q)}\setminus K_N^{(q)}\) arising at endpoints of reduced denominator
\(N\).

On \(K_N^{(q)}(\theta_-,\theta_+)\), Lemma~\ref{lem:intervals} gives
\(\left\lfloor\ell_N(q)\rho/\ell_e(q)\right\rfloor=n_N^{(q)}(\theta_-)\), so the residue in
Lemma~\ref{lem:pure} is
\[
 \delta=[n_N^{(q)}(\theta_-)-\rho]_{q^N}
       =[n_N^{(q)}(\theta_-)-\sigma_N^{(q)}(\theta_-)]_{q^N}.
\]
Since \(n_N^{(q)}(\theta_-)-\sigma_N^{(q)}(\theta_-)\equiv-1\pmod q\), this residue
is not in \(\{0,\ldots,q-2\}\).  Hence no point of \(K_N^{(q)}\) is excluded.
Together with the preceding exclusion of \(J_N^{(q)}\setminus K_N^{(q)}\), this
proves the two assertions.
\end{proof}

For the remainder of this section, let \(N,\theta_-,\theta_+\) be as in
Lemma~\ref{lem:farey-localization}.  Let
\(v_N:=[\lfloor\rho/q^N\rfloor]_q\) for \(\rho=[r]_{\ell_e(q)}\); this is the
base-\(q\) digit of \(\rho\) in position \(N\), namely,
\(\displaystyle \rho=\sum_{j\geqslant0}v_jq^j\) with \(0\leqslant v_j<q\).
The second condition in \eqref{eq:farey-condition} says precisely that
\(v_0=1\) and \(v_k=\varepsilon_k(\theta_-)\) for all
\(1\leqslant k<N\).  We shall prove that the hypotheses of
Lemma~\ref{lem:farey-localization} also force
\(v_N=\varepsilon_N(\theta_-)\).  This is the only point in the necessity argument for \(q>2\) at which odd
characteristic and characteristic two require separate arguments: Proposition~\ref{prop:odd-selector}
handles odd characteristic, while Proposition~\ref{prop:two-selector} handles
characteristic two.

Since \(q=p^m\), we group the base-\(p\) digits into consecutive blocks of
length \(m\), numbered in the same way; block \(N\) records the base-\(q\)
digit in position \(N\).  For a nonzero integer \(u\), let \(\nu_p(u)\)
denote the largest integer \(j\geqslant0\) such that \(p^j\mid u\).
For a condition \(\mathcal C\), write
\(\mathbf1_{\mathcal C}\in\{0,1\}\) for its indicator, namely,
\(\mathbf1_{\mathcal C}=1\) if \(\mathcal C\) holds, and \(0\) otherwise.
By Lemma~\ref{lem:farey-localization}, put
\(\gamma:=\ell_N(q)\rho/\ell_e(q)-n_N^{(q)}(\theta_-)\), so
\(0\leqslant\gamma<1\).  For every integer \(1\leqslant C<q\), put
\[
 t_C:=\left\lfloor\frac{(C\ell_N(q)+1)\rho}{\ell_e(q)}\right\rfloor
     -Cn_N^{(q)}(\theta_-)
     =\left\lfloor C\gamma+\frac{\rho}{\ell_e(q)}\right\rfloor.
\]
Then \(0\leqslant t_C\leqslant C\).  Moreover,
\(1\leqslant C\ell_N(q)+1\leqslant q^N<\ell_e(q)\), so
Lemma~\ref{lem:power-sum} applies at \(\tau=C\ell_N(q)+1\).

\subsection*{Odd characteristic}
\begin{proposition}[Odd-characteristic digit forcing]
\label{prop:odd-selector}
Suppose that \(q=p^m\) with \(p\) odd, that \(X^r(X^{q-1}+a)\)
permutes \(\mathbb F_{q^e}\), and that condition~\eqref{eq:farey-condition} holds.  Then
\(v_N=\varepsilon_N(\theta_-)\).
\end{proposition}

\begin{proof}
Suppose, to the contrary, that
\(d=v_N-\varepsilon_N(\theta_-)\ne0\).  Put
\(j=\nu_p(d)\), \(b=p^j\), and \(d=bw\), where \(p\nmid w\).
Since \(-1\leqslant d\leqslant q-1\) and \(d\ne0\), we have \(j<m\).

First consider the boundary case \(j=0\).  Then \(b=1\) and \(d=w\).
Take \(\tau=\ell_N(q)+1\).  Since \(0\leqslant t_1\leqslant1\), the lowest
base-\(q\) block leaves a unique possible coefficient exponent: it is \(q-2\)
when \(t_1=0\) and \(0\) when \(t_1=1\).  Reducing the corresponding lower
index modulo \(q^{N+1}\), using \((q-1)\ell_e(q)\equiv-1\pmod{q^{N+1}}\),
shows that its digit in block \(N\) is
\[
 [q-1+2w]_q.
\]
Its least significant base-\(p\) digit is \([2w-1]_p\ne p-1\), because
\(p\) is odd and \(p\nmid w\).  The lowest block already satisfies Lucas'
digit inequalities, and every other relevant upper digit is \(p-1\).
Thus \(P_{\rho,\ell_N(q)+1}\) is a nonzero monomial, contradicting
Lemma~\ref{lem:power-sum}.  Hence \(j\geqslant1\).

Now \(p\mid b\).  For \(c\in\{1,p-1\}\), set \(C=cb\) and
\(\tau=C\ell_N(q)+1\).  Then \(C<q\), so the common setup above applies,
and moreover \(p\mid C\).  Put \(t=t_C\) and, for
\(0\leqslant\ell\leqslant c\), put
\[
 i_\ell=[t-C-1+\ell b]_q,\qquad
 \delta_\ell=\mathbf1_{\,t+\ell b\leqslant C}.
\]
The upper index has lowest base-\(q\) block \(C\), block \(N\) equal to
\(q-1-C\), and all other blocks equal to \(q-1\); hence the possible lowest
blocks of a lower index are \(\ell b\), \(0\leqslant\ell\leqslant c\), with
coefficient exponent \(i_\ell\).  The value \(i_\ell=q-1\) is outside
\(0,\ldots,q-2\) and is omitted.  For a remaining index put
\[
 \begin{aligned}
 A_\ell&=(C\ell_N(q)+1)\rho+(i_\ell-Cn_N^{(q)}(\theta_-)-t)\ell_e(q),\\
 E_\ell&=(q-1-C)(\sigma_N^{(q)}(\theta_-)-1)+Cn_N^{(q)}(\theta_-)
          +q(1-\delta_\ell)-\ell b.
 \end{aligned}
\]
The bounds and congruences in Lemma~\ref{lem:labels} give
\(-q<E_\ell<q^N\), and reduction modulo \(q^{N+1}\) gives
\[
 A_\ell\equiv-\ell_N(q)E_\ell
 +q^N\bigl(\sigma_N^{(q)}(\theta_-)+(C+1)v_N+i_\ell-Cn_N^{(q)}(\theta_-)-t\bigr)
 \pmod{q^{N+1}}.
\]
Put
\[
 \lambda=\frac{(q-1-C)(\sigma_N^{(q)}(\theta_-)-1)+Cn_N^{(q)}(\theta_-)-C\varepsilon_N(\theta_-)}{q-1}\in\mathbb Z.
\]
Then \(\lambda\equiv C\varepsilon_N(\theta_-)\pmod q\) and
\(\lceil\ell_N(q)E_\ell/q^N\rceil=\lceil E_\ell/(q-1)\rceil
=\lambda+\varepsilon_N(\theta_-)+1-\delta_\ell\).
Here the last equality uses \(\delta_\ell=1\Rightarrow\ell b\leqslant C-1\)
and \(\delta_\ell=0\Rightarrow1\leqslant\ell b\leqslant C\).  Hence the digit in
block \(N\) of \(A_\ell\) is
\[
 \bigl[q-C+\ell b-1+\delta_\ell+(C+1)(v_N-\varepsilon_N(\theta_-))\bigr]_q.
\]
Since \(v_N-\varepsilon_N(\theta_-)=bw\) and \(b=p^j\), the relevant
base-\(p\) digit before the deviation is added is
\([\ell-c-1+\delta_\ell]_p\).  Moreover,
\((C+1)bw/b=(C+1)w\equiv w\pmod p\), because \(p\mid C\).  Hence the
relevant base-\(p\) digit is
\begin{equation}
 \kappa_\ell=[w-c+\ell-1+\delta_\ell]_p.
 \label{eq:kappa-odd}
\end{equation}
At every other base-\(p\) position, the corresponding upper digit is \(p-1\), and
\(\displaystyle\binom{p-1}{u}\equiv(-1)^u\pmod p\) for \(0\leqslant u<p\).

For \(c=1\), if \(t\in\{0,b\}\), only one coefficient index remains,
and its relevant digit is \([w-1]_p\ne p-1\); hence the coefficient polynomial
is a nonzero monomial.  If \(0<t<b\), both coefficient indices occur and have
the same nonzero Lucas factor at the relevant base-\(p\) digit, and their
exponents differ by \(q-b\).  Their digits in position \(j\) of the lowest
base-\(q\) block are \(0\) and \(1\).  Since \(p\) is odd, an integer has the
same parity as its base-\(p\) digit sum; the two lower indices differ by the
even number \((q-b)\ell_e(q)\).  Thus their total digit-sum parities agree,
so the remaining contributions from positions with upper digit \(p-1\) have opposite sign.  Thus
\begin{equation}
 P_{\rho,b\ell_N(q)+1}(X)=C_-X^{t-1}(1-X^{q-b}),
 \qquad C_-\ne0.
 \label{eq:minus-probe}
\end{equation}

For \(c=p-1\), \eqref{eq:kappa-odd} is zero precisely when
\(\ell+\delta_\ell\equiv-w\pmod p\).  As \(\ell\) runs from \(0\) to \(p-1\),
the sequence \(\ell+\delta_\ell\) contains every nonzero residue, with only
the transition value possibly repeated.  Since \(-w\ne0\), one or two
indices remain.  If the omitted value \(i_\ell=q-1\) corresponds to a
solution, it is either the irrelevant zero residue at \(t=0\) or the repeated
transition value, so another solution remains.  In the two-term case the
exponents differ by \(q-b\).  As above, the total base-\(p\) digit-sum
parities agree; every allowed position has upper digit \(p-1\), except for the
forced zero digit, so the two coefficients have the same sign.  Consequently
the coefficient polynomial is a nonzero monomial or
\begin{equation}
 P_{\rho,(p-1)b\ell_N(q)+1}(X)
 =C_+X^i(1+X^{q-b}),\qquad C_+\ne0.
 \label{eq:plus-probe}
\end{equation}

If either coefficient polynomial is a monomial,
Lemma~\ref{lem:power-sum} contradicts the permutation hypothesis.  Otherwise
evaluating \eqref{eq:minus-probe} and \eqref{eq:plus-probe} at
\(a^{-\ell_e(q)}\) would force
\((a^{-\ell_e(q)})^{q-b}=1\) and \((a^{-\ell_e(q)})^{q-b}=-1\),
impossible in odd characteristic.  Hence
\(v_N=\varepsilon_N(\theta_-)\).
\end{proof}

\subsection*{Characteristic two}
Assume now \(q=2^m\), \(m\geqslant2\).  To retain
coefficient information modulo \(4\) that disappears modulo \(2\), we lift the relevant
coefficient identities to an auxiliary ring of characteristic \(4\).  For odd
\(1\leqslant\tau<\ell_e(q)\), the integer \(H_\tau\) is even.  We retain the same coefficients modulo \(4\), encoded by
\[
 P^{(4)}_{\rho,\tau}(X):=\sum_{i=0}^{q-2}
 \displaystyle\binom{H_\tau}{y_\tau+i\ell_e(q)}X^i
 \in(\mathbb Z/4\mathbb Z)[X].
\]

Put \(Q:=q^e=2^{me}\).  Choose a monic irreducible
\(g\in\mathbb F_2[U]\) of degree \(me\), a monic lift
\(\widetilde g\in(\mathbb Z/4\mathbb Z)[U]\), set
\(\mathcal R:=(\mathbb Z/4\mathbb Z)[U]/(\widetilde g)\), and fix an
isomorphism \(\mathcal R/2\mathcal R\cong\mathbb F_Q\).
For \(x\in\mathbb F_Q\), choose any lift \(u\in\mathcal R\) and put
\(\widehat{x}:=u^Q\).
Since \(Q\) is even and \(4=0\) in \(\mathcal R\),
\((u+2w)^Q=u^Q\); hence \(\widehat{x}\) is well defined, reduces to \(x\),
and satisfies \(\widehat{x}^{\,Q}=\widehat{x}\).  If \(z\in\mathcal R\) is any
lift of \(x\) with \(z^Q=z\), then using \(z\) in the definition gives
\(\widehat{x}=z^Q=z\).  Thus \(\widehat{x}\) is the unique lift
of \(x\) fixed by the \(Q\)-power map.  Using \(uv\) as a lift of \(xy\) then
gives \(\widehat{xy}=\widehat{x}\widehat{y}\).  Also, if
\(b\in\mathcal R\) has nonzero residue \(\bar b\in\mathcal R/2\mathcal R\), lift
\(\bar b^{-1}\) to \(c\); then
\(bc=1+2w\) and \((1+2w)^{-1}=1-2w\).  Thus \(b\) is a unit, so
\(\widehat{x}\) is a unit exactly when \(x\ne0\).

\begin{lemma}[Mod-\(4\) lift of the Hermite sums]
\label{lem:mod4-lift}
If \(X^r(X^{q-1}+a)\) permutes \(\mathbb F_{q^e}\) and
\(1\leqslant\tau<\ell_e(q)\) is odd, then
\[
 P^{(4)}_{\rho,\tau}\bigl(\widehat{a^{-\ell_e(q)}}\bigr)=0
 \qquad\text{in }\mathcal R.
\]
In particular, \(P^{(4)}_{\rho,\tau}(X)=2X^i\) is impossible.
\end{lemma}

\begin{proof}
Put \(\mathcal T=\{\widehat{x}:x\in\mathbb F_Q\}\).  By multiplicativity, reduction identifies \(\mathcal T^\times\) with the cyclic
group \(\mathbb F_Q^\times\) of order \(q^e-1\).  Hence, for every positive
integer \(d\), the geometric-series identity gives
\[
 \sum_{t\in\mathcal T}t^d=
 \begin{cases}
 q^e-1,&(q^e-1)\mid d,\\
 0,&(q^e-1)\nmid d.
 \end{cases}
\]
Indeed, when \((q^e-1)\nmid d\), the denominator in the geometric series is a
unit because its residue is nonzero.

Put \(\widehat f(X)=X^r(X^{q-1}+\widehat a)\).  If two elements of
\(\mathcal R\) are congruent modulo \(2\mathcal R\), write them as \(u\) and
\(u+2w\).  For even \(n>0\), the binomial expansion of \((u+2w)^n\) has
linear correction term \(2n u^{n-1}w=0\), while every higher correction term
is divisible by \(4\); hence \((u+2w)^n=u^n\).  Thus elements congruent modulo \(2\mathcal R\) have the same positive even
powers.  Since \(H_\tau\) is even and
\(\widehat f(\widehat x)\) reduces to \(f(x)\),
\[
 \sum_{t\in\mathcal T}\widehat f(t)^{H_\tau}
 =\sum_{x\in\mathbb F_Q}\widehat{f(x)}^{\,H_\tau}
 =\sum_{x\in\mathbb F_Q}\widehat{x}^{\,H_\tau}=0,
\]
where the second equality uses that \(f\) permutes \(\mathbb F_Q\), and the
last follows from \(0<H_\tau<q^e-1\) and the preceding identity.

Since \(q^e-1\equiv-1\pmod 4\), expanding the first sum and using the same
divisibility condition as in Lemma~\ref{lem:power-sum} gives
\[
 0=-\widehat a^{\,H_\tau-y_\tau}
 P^{(4)}_{\rho,\tau}(\widehat a^{-\ell_e(q)}).
\]
By multiplicativity,
\(\widehat a^{-\ell_e(q)}=\widehat{a^{-\ell_e(q)}}\).  The evaluation point is a
unit and \(2\ne0\) in \(\mathcal R\), so a polynomial \(2X^i\) cannot vanish
there.
\end{proof}

\smallskip
\noindent\textbf{Kummer's theorem.} Recall that, for a nonzero integer \(u\),
\(\nu_2(u)\) denotes the largest integer \(j\geqslant0\) such that \(2^j\mid u\).
In the case \(p=2\), Kummer's theorem states that, for
\(0\leqslant A\leqslant n\), \(\nu_2\displaystyle\binom nA\) equals the number
of carries in the base-\(2\) addition of \(A\) and \(n-A\); see
\cite[Theorem~10.2.2]{AndreescuAndrica2009}, and also the original paper
\cite[p.~116]{Kummer1852}.  Equivalently, it is the number of borrows in the base-\(2\) subtraction of
\(A\) from \(n\).  Thus zero, one, or at least two borrows correspond,
respectively, to an odd coefficient, a coefficient
\(2\pmod 4\), or one divisible by \(4\).  For
\(A\notin\{0,\ldots,n\}\), we use the convention \(\displaystyle\binom nA=0\).

\begin{proposition}[Characteristic-two digit forcing]
\label{prop:two-selector}
Suppose that \(q=2^m\) with \(m\geqslant2\), that \(X^r(X^{q-1}+a)\)
permutes \(\mathbb F_{q^e}\), and that condition~\eqref{eq:farey-condition} holds.  Then
\(v_N=\varepsilon_N(\theta_-)\).
\end{proposition}

\begin{proof}
We treat three cases: nonzero even deviations; odd deviations at \(N=1\);
and odd deviations for \(N\geqslant2\), where two endpoint cases require the
mod-\(4\) lift.

\emph{Even deviations.} First suppose \(d=v_N-\varepsilon_N(\theta_-)\ne0\) is even and put
\(b=2^{\nu_2(d)}\).  Then \(b\leqslant q/2\), so
\(\tau=b\ell_N(q)+1\) is odd and admissible.
The lowest base-\(q\) block leaves only the coefficient exponents
\(i_0=[t_b-b-1]_q\) and \(i_b=[t_b-1]_q\).  Reducing modulo
\(q^{N+1}\) and using the digit congruence derived from
\eqref{eq:farey-condition} together with
\((q-1)\ell_e(q)\equiv-1\pmod{q^{N+1}}\) gives the digits in block \(N\)
\[
 [q-b+(b+1)d]_q,\qquad [q-1+(b+1)d]_q.
\]
Since \(d/b\) is odd and \(2b\mid q\), we have
\((b+1)d\equiv b\pmod{2b}\); hence the bit in position \(\log_2 b\) is zero for both candidates.
Thus, if \(t_b\in\{0,b\}\), the coefficient polynomial is a monomial,
which is impossible by Lemma~\ref{lem:power-sum}.  Hence \(0<t_b<b\), and the two surviving
terms give
\begin{equation}
 X^{t_b-1}(1+X^{q-b}).
 \label{eq:two-even-row}
\end{equation}
If \(b=q/2\), then
\((a^{-\ell_e(q)})^{q-1}=a^{-(q-1)\ell_e(q)}=a^{-(q^e-1)}=1\), so
\(a^{-\ell_e(q)}\in\mathbb F_q^{\times}\).  Evaluating
\eqref{eq:two-even-row} at \(a^{-\ell_e(q)}\) gives
\((a^{-\ell_e(q)})^{q/2}=1\); since \((q/2,q-1)=1\), this forces
\(a^{-\ell_e(q)}=1\), contradicting \(a^{\ell_e(q)}\ne1\).  If \(b<q/2\), then \(2b<q\), so the
common setup also applies with \(C=2b\).  The formula for \(t_C\), together with
\(0<t_b<b\) and \(0<\rho/\ell_e(q)<1\), implies
\(0<t_{2b}<2b\).  Hence both coefficient indices for
\(2b\ell_N(q)+1\) occur.  For \(C=2b\), their block-\(N\) digits are
\([q-2b+(2b+1)d]_q\) and \([q-1+(2b+1)d]_q\).  Since
\(4b\mid q\) and \(d/b\) is odd,
\((2b+1)d\equiv b\) or \(3b\pmod{4b}\).
Thus the bit in position \(\log_2(2b)\) is complementary on the two
candidates.  As this is the unique zero bit of the upper block
\(q-1-2b\), exactly one coefficient is nonzero modulo \(2\), again
contradicting Lemma~\ref{lem:power-sum}.  Thus every nonzero even deviation is impossible.

\emph{Odd deviations at \(N=1\).} At the first level
\(N=1\), write \(\rho\equiv1+v_1q\pmod{q^2}\).  For \(\tau=1\), the
upper index \(q^e-q\) has its lowest \(m\) binary digits zero and all
higher digits one.  Exactly one borrow occurs only when the lower
index is congruent to \(q/2\pmod q\) and the next binary digit is
zero.  Hence
\[
 P^{(4)}_{\rho,1}=
 \begin{cases}
 2X^{q/2-1},&v_1\text{ odd},\\
 0,&v_1\text{ even}.
 \end{cases}
\]
Lemma~\ref{lem:mod4-lift} excludes odd \(v_1\), while the preceding
even-deviation argument excludes every nonzero even \(v_1\).  Thus
\(v_1=0=\varepsilon_1(\theta_-)\).

\emph{Odd deviations for \(N\geqslant2\).} Now suppose
\(v_N-\varepsilon_N(\theta_-)\) is odd.  Put
\[
 \tau=\ell_N(q)+2,\qquad
 c=\left\lfloor\gamma+\frac{2\rho}{\ell_e(q)}\right\rfloor\in\{0,1,2\}.
\]
Here \(\tau\) is odd and \(\tau<\ell_e(q)\).  The upper base-\(q\) blocks
are \((2,q-2,q-1,\ldots,q-1,q-2,q-1,\ldots)\), with the second
\(q-2\) in position \(N\).  The lowest block leaves
candidates with coefficient exponents
\(i_0=[c-3]_q\) and \(i_2=[c-1]_q\).  The congruences \eqref{eq:strong-state} show that block 1 leaves only
\(i_2\) when
\(\varepsilon_N(\theta_-)=0\), and only \(i_0\) when \(\varepsilon_N(\theta_-)=1\).  The
relevant base-\(p\) bit in block \(N\) is then zero because
\(v_N-\varepsilon_N(\theta_-)\) is odd.  Hence the coefficient
polynomial modulo \(2\) is a monomial except in the two endpoint cases
\((\varepsilon_N(\theta_-),c)=(0,0)\) and \((1,2)\).  Outside these endpoints,
Lemma~\ref{lem:power-sum} gives an immediate contradiction.

At the two endpoints the coefficient polynomial modulo \(2\) is zero.
We claim that the corresponding coefficient polynomials modulo \(4\) are
\begin{equation}
 2X^{q-2},\qquad 2.
 \label{eq:endpoint-mod4}
\end{equation}
Indeed, by Kummer's theorem above, a coefficient nonzero modulo \(4\) but
zero modulo \(2\) must correspond to exactly one borrow.  For a candidate with exponent \(i\), put
\[
 A_i=(\ell_N(q)+2)\rho+(i-n_N^{(q)}(\theta_-)-c)\ell_e(q).
\]
Since \(N\geqslant2\), the congruences \(\rho\equiv1\pmod{q^2}\),
\(n_N^{(q)}(\theta_-)\equiv q\varepsilon_N(\theta_-)\pmod{q^2}\), and
\(\ell_N(q)\equiv\ell_e(q)\equiv1+q\pmod{q^2}\) give
\(A_i\equiv(3-c+i)+q(1-c+i-\varepsilon_N(\theta_-))\pmod{q^2}\).
If exactly one binary borrow occurs, an even lower index can only have
lowest base-\(q\) block \(0\) or \(2\), while an odd lower index must
have lowest block \(1\).  The only candidates requiring inspection are
therefore
\[
\begin{array}{c|c|c}
(\varepsilon_N(\theta_-),c)&i&(\text{block }0,\text{ block }1)\\ \hline
(0,0)&q-3&(0,q-1)\\
(0,0)&q-2&(1,0)\\
(1,2)&1&(2,q-1)\\
(1,2)&0&(1,q-2).
\end{array}
\]
For the two even candidates, the upper block 1 is \(q-2\) whereas the
lower block 1 is \(q-1\); the borrow therefore traverses all
\(m\geqslant2\) binary positions of that block, so the valuation is at least
\(2\).

For the two odd candidates, the borrow starts at bit \(0\) and stops at bit
\(1\).  To check that no further borrow occurs, write
\(s=\sigma_N^{(q)}(\theta_-)\), \(n=n_N^{(q)}(\theta_-)\), and \(v=v_N\).
For the endpoint \((\varepsilon_N(\theta_-),c)=(0,0)\), put
\(k_0=(s+q-2-n)/(q-1)\); for \((\varepsilon_N(\theta_-),c)=(1,2)\), put
\(k_1=(s-n)/(q-1)\).  Lemma~\ref{lem:labels} shows that \(k_0,k_1\) are odd
integers, and its bounds give
\(0<2s-k_0<q^N\) and
\(-q^N<2s-k_1-2\ell_N(q)<0\).  Expanding the corresponding \(A_i\) modulo
\(q^{N+1}\), the block-\(N\) digits of the odd candidates \(i=q-2\) and
\(i=0\) are, respectively,
\[
 [3v+q-2-n+k_0]_q,
 \qquad
 [3v-n+k_1-3]_q.
\]
In the first endpoint \(v\) is odd, while in the second it is even; also
\(n\) is even and \(k_0,k_1\) are odd.  Hence both displayed digits are
even.  The upper block \(q-2\) has its only zero binary digit in the least
significant position, so these even lower digits create no borrow in block
\(N\); all its remaining binary digits are \(1\).  Thus the unique valuation-one
terms have exponents \(q-2\) and \(0\), respectively, and coefficient
\(2\pmod 4\); every other candidate is zero modulo \(4\).  This proves
\eqref{eq:endpoint-mod4}.
Lemma~\ref{lem:mod4-lift} gives the contradiction.  Therefore odd
deviations are impossible as well, and \(v_N=\varepsilon_N(\theta_-)\).
\end{proof}

Together, Lemma~\ref{lem:farey-localization} and
Propositions~\ref{prop:odd-selector} and~\ref{prop:two-selector} give the one-step
mechanism in the second stage: condition~\eqref{eq:farey-condition} first localizes
the normalized residue to the stable interval \(K_N^{(q)}\), and the permutation
hypothesis then forces the actual next digit to be
\(v_N=\varepsilon_N(\theta_-)\).  Section~\ref{sec:propagation} iterates this
mechanism through the Farey levels and then passes to the third, purely arithmetic
stage.

\section{Farey propagation and the inverse congruence}
\label{sec:propagation}

We first complete the second stage by iterating the one-step mechanism through
levels \(1,\ldots,e-1\).  Proposition~\ref{prop:propagation} carries the Farey cell
and the digit congruence from one level to the next, and
Corollary~\ref{cor:quotients} records the resulting binary quotient recursion.  We
then complete the third stage by converting that recursion into the inverse
congruence through cyclic rotations.

\begin{proposition}[Farey propagation]
\label{prop:propagation}
Let \(q=p^m>2\), let \(e\geqslant2\), and suppose that
\(X^r(X^{q-1}+a)\) permutes \(\mathbb F_{q^e}\).  Put \(\rho=[r]_{\ell_e(q)}\),
and let \(1\leqslant N<e\).  Then:
\begin{enumerate}[label=\textup{(\arabic*)},leftmargin=*,itemsep=1pt,topsep=2pt]
\item There exists a unique adjacent pair \(\theta_-<\theta_+\) in \(F_N\) such that
\[
 \frac{\rho}{\ell_e(q)}\in K_N^{(q)}(\theta_-,\theta_+),\qquad
 \rho\equiv\sigma_N^{(q)}(\theta_-)\pmod{q^N}.
\]
\item There exists a unique adjacent pair \(\theta_-'<\theta_+'\) in \(F_{N+1}\) such that
\[
 \frac{\rho}{\ell_e(q)}\in J_{N+1}^{(q)}(\theta_-',\theta_+').
\]
\end{enumerate}
\end{proposition}

\begin{proof}
We prove \textup{(1)} by induction on \(N\), with \textup{(2)} providing the
passage from level \(N\) to level \(N+1\).  For \(N=1\), Lemma~\ref{lem:pure}
gives \(1\leqslant\rho<\ell_e(q)\) and \(\rho\equiv1\pmod q\).  Since
\(K_1^{(q)}(0,1)=[0,1)\) and \(\sigma_1^{(q)}(0)=1\), the unique pair
\(0<1\) has the properties in \textup{(1)}.

Assume that \textup{(1)} holds at level \(N<e\), with pair
\(\theta_-<\theta_+\).  Corollary~\ref{cor:one-step-intervals} shows that
\(K_N^{(q)}(\theta_-,\theta_+)\) is either a single
\(J_{N+1}^{(q)}\)-interval or the disjoint union of two such intervals.  Hence
\(\rho/\ell_e(q)\) lies in a unique interval
\(J_{N+1}^{(q)}(\theta_-',\theta_+')\), which proves \textup{(2)}.

If \(N=e-1\), there is nothing further to prove.  Suppose now that \(N<e-1\).
Because \(K_N^{(q)}(\theta_-,\theta_+)\subseteq
J_N^{(q)}(\theta_-,\theta_+)\), condition~\eqref{eq:farey-condition} holds.  Let
\(v_N=[\lfloor\rho/q^N\rfloor]_q\).  Propositions~\ref{prop:odd-selector}
and~\ref{prop:two-selector} give
\(v_N=\varepsilon_N(\theta_-)\).  Therefore
\[
 \rho\equiv\sigma_N^{(q)}(\theta_-)+\varepsilon_N(\theta_-)q^N
 =\sigma_{N+1}^{(q)}(\theta_-)\pmod{q^{N+1}},
\]
where the last equality is by Lemma~\ref{lem:labels}.

For the pair \(\theta_-'<\theta_+'\) supplied by \textup{(2)}, one has
\(\theta_-\leqslant\theta_-'<\theta_+\).  Hence Lemma~\ref{lem:labels} gives
\(\sigma_N^{(q)}(\theta_-')=\sigma_N^{(q)}(\theta_-)\) and
\(\varepsilon_N(\theta_-')=\varepsilon_N(\theta_-)\), so
\[
 \rho\equiv\sigma_{N+1}^{(q)}(\theta_-')\pmod{q^{N+1}}.
\]
Applying Lemma~\ref{lem:farey-localization} at level \(N+1\) yields
\[
 \frac{\rho}{\ell_e(q)}\in K_{N+1}^{(q)}(\theta_-',\theta_+'),
\]
which is \textup{(1)} at level \(N+1\) and completes the induction.
\end{proof}

\begin{corollary}
\label{cor:quotients}
Let \(q>2\) be a prime power and \(e\geqslant2\), and suppose that
\(X^r(X^{q-1}+a)\) permutes \(\mathbb F_{q^e}\).  Put \(\rho=[r]_{\ell_e(q)}\) and
\(s_d=\lfloor\ell_d(q)\rho/\ell_e(q)\rfloor\) for \(0\leqslant d\leqslant e\).
Then \(1\leqslant\rho<\ell_e(q)\), \(\rho\equiv1\pmod q\), and
\[
 s_{d+1}-qs_d\in\{0,q\}\quad\text{for all }0\leqslant d\leqslant e-2.
\]
\end{corollary}

\begin{proof}
Apply Proposition~\ref{prop:propagation} through level \(e-1\).
For \(N\geqslant2\), let \(\theta_-<\theta_+\) be the pair in
Proposition~\ref{prop:propagation}\textup{(1)} at level \(N\).  By
Corollary~\ref{cor:one-step-intervals}, \(\theta_-\) lies in the same
level-\((N-1)\) half-open Farey interval as the left endpoint of the preceding
pair.  The \(K_N^{(q)}\)-membership in
Proposition~\ref{prop:propagation}\textup{(1)} and
Lemma~\ref{lem:labels} give \(s_N=n_N^{(q)}(\theta_-)\); applying the same
lemma to the containing level-\((N-1)\) cell gives
\(s_{N-1}=n_{N-1}^{(q)}(\theta_-)\).  Hence
\(s_N-qs_{N-1}=q\varepsilon_N(\theta_-)\in\{0,q\}\).
Initially, \(s_0=s_1=0\), and Lemma~\ref{lem:pure} gives
\(\rho\equiv1\pmod q\).  For \(e=2\), only \(d=0\) occurs.
\end{proof}

\subsection*{Ordering the cyclic rotations}
At this point the Hermite, Lucas, and Farey parts of the argument are complete.  It remains only
to carry out the third stage: convert the binary quotient recursion in
Corollary~\ref{cor:quotients} into the inverse congruence.  This step is purely
arithmetic and requires only an integer base \(q>2\).  For related lexicographic
arrays, see \cite{Zamboni26}; the argument below is self-contained and uses only elementary number theory.

\begin{proposition}
\label{prop:rotations}
Let \(q>2\) and \(e\geqslant2\) be integers, and put
\(\ell_j(q)=(q^j-1)/(q-1)\) for \(j\geqslant0\).  Suppose
\[
 1\leqslant\rho<\ell_e(q),\qquad\text{and}\qquad \rho\equiv1\pmod q,
\]
and put \(s_d=\lfloor\ell_d(q)\rho/\ell_e(q)\rfloor\) for
\(0\leqslant d\leqslant e\).  If
\[
 s_{d+1}-qs_d\in\{0,q\}\quad\text{for all }0\leqslant d\leqslant e-2,
\]
then
\(\rho\ell_h(q)\equiv1\pmod{\ell_e(q)}\) for some
\(1\leqslant h<e\) with \(\gcd(h,e)=1\).
\end{proposition}

\begin{proof}
Note that \(s_0=s_1=0\).  For \(1\leqslant d\leqslant e-2\), set
\(\delta_d=(s_{d+1}-qs_d)/q\in\{0,1\}\).  (The range is empty when
\(e=2\).)  Iterating from \(s_1=0\) gives
\[
 s_d=\sum_{i=1}^{d-1}\delta_iq^{d-i}
 \quad\text{for all }1\leqslant d\leqslant e-1.
\]
Also \(q\ell_{e-1}(q)=\ell_e(q)-1\).  Writing \(\rho=qm+1\) and using
\(0<\rho<\ell_e(q)\), we obtain
\[
 s_{e-1}
 =\left\lfloor\frac{\rho}{q}-\frac{\rho}{q\ell_e(q)}\right\rfloor
 =m=\frac{\rho-1}{q}.
\]
Consequently
\[
 \rho=1+\sum_{i=1}^{e-2}\delta_iq^{e-i}.
\]
Thus \(\rho\) is the base-\(q\) value of a length-\(e\) word \(W\) with
binary digits, ending in \(01\): explicitly, \(W=01\) if \(e=2\), and
\(W=\delta_1\cdots\delta_{e-2}01\) if \(e\geqslant3\).

Let \(L\) denote left cyclic rotation, and let \(\rho_d\) be the base-\(q\)
value of \(L^dW\).  Thus \(\rho_0=\rho\).  Since \(W\) has only the digits
\(0,1\), is nonzero, and is not the all-one word, we have
\(0<\rho_d<\ell_e(q)<q^e-1\).  The first \(d\) digits of \(W\) have
base-\(q\) value \(s_d+\delta_d\) for \(1\leqslant d\leqslant e-2\).
Moving these digits cyclically to the end therefore gives
\[
 q^d\rho=(s_d+\delta_d)(q^e-1)+\rho_d
 \quad\text{for all }1\leqslant d\leqslant e-2,
 \qquad
 q^{e-1}\rho=s_{e-1}(q^e-1)+\rho_{e-1}.
\]
For \(1\leqslant d\leqslant e-2\), subtracting \(\rho\) and dividing by
\(q^e-1\) gives
\[
 \frac{(q^d-1)\rho}{q^e-1}
 =s_d+\delta_d+\frac{\rho_d-\rho}{q^e-1}.
\]
Since \(s_d=\lfloor (q^d-1)\rho/(q^e-1)\rfloor\), taking floors gives
\[
 0=\delta_d+\left\lfloor\frac{\rho_d-\rho}{q^e-1}\right\rfloor.
\]
Thus \(\rho_d<\rho\) exactly when \(\delta_d=1\).  For \(d=e-1\), the same
comparison gives \(\rho_{e-1}\geqslant\rho\).  Since the last digit of
\(L^dW\) is \(\delta_d\) for \(1\leqslant d\leqslant e-2\), and is \(0\)
for \(d=e-1\), we have proved
\begin{equation}
 \rho_d<\rho
 \quad\Longleftrightarrow\quad
 L^dW\text{ ends in }1
 \quad\text{for all }1\leqslant d<e.
 \label{eq:rotation-order}
\end{equation}
We shall use only the rotation-order criterion~\eqref{eq:rotation-order} from now on.
In particular, the \(e\) cyclic rotations are distinct: if \(L^dW=W\) for some
\(1\leqslant d<e\), then \(L^dW\) ends in \(1\), so
\eqref{eq:rotation-order} would give \(\rho_d<\rho\), a contradiction.

Let \(k\) be the number of \(1\)'s in \(W\).  Since \(W\) ends in \(01\),
we have \(1\leqslant k<e\).  Sort the \(e\) rotations \(L^dW\),
\(0\leqslant d<e\), in increasing lexicographic order, equivalently by their
base-\(q\) values \(\rho_d\), and assign them lexicographic ranks
\(0,\ldots,e-1\), starting with rank \(0\) for the smallest rotation.
Exactly \(k-1\) nontrivial rotations end in \(1\), and by
\eqref{eq:rotation-order} these are precisely the rotations smaller than \(W\).
Hence \(W\) has rank \(k-1\).  It follows that the rotations ending in \(1\) have
ranks \(0,\ldots,k-1\), while the rotations beginning in \(1\) have ranks
\(e-k,\ldots,e-1\).

Right cyclic rotation sends the rotations ending in \(1\) to those beginning in \(1\)
and preserves their relative order; within a fixed last-digit class,
lexicographic order is determined by the first \(e-1\) digits.  The same holds
for the rotations ending in \(0\).  For an integer \(z\), let \([z]_e\) denote its
least nonnegative residue modulo \(e\).  Thus right rotation sends rank \(j\) to \([j-k]_e\),
and left rotation sends rank \(j\) to \([j+k]_e\).  Since the rotations are
distinct, left rotation is an \(e\)-cycle.  Therefore translation by \(k\)
on \(\mathbb Z/e\mathbb Z\) has order \(e\), and hence \(\gcd(k,e)=1\).

Choose \(1\leqslant h<e\) with \(hk\equiv1\pmod e\), and put
\(W'=L^hW\).  Since \(W\) has rank \(k-1\), the word \(W'\) has rank
\(k\), so it is the immediate successor of \(W\).  Number the digit
positions from \(0\) to \(e-1\) from left to right.  The rank of \(L^iW\) is
\([(i+1)k-1]_e\), while the rank of \(L^iW'\) is \([(i+1)k]_e\).  A rotation
begins in \(1\) exactly when its rank is at least \(e-k\).  Hence \(W_i\)
and \(W'_i\) differ only when \([(i+1)k]_e=e-k\) or \(0\).  Since
\(\gcd(k,e)=1\), these two cases occur exactly for \(i=e-2\) and \(i=e-1\),
respectively.  Thus \(W\) and \(W'\) agree in their first \(e-2\) digits
and end in \(01\) and \(10\), respectively.  Therefore the base-\(q\) value
of \(W'\) is \(\rho+q-1\).  Since \(W'=L^hW\),
\[
 q^h\rho\equiv\rho+q-1\pmod{q^e-1}.
\]
Using \(q^h-1=(q-1)\ell_h(q)\) and
\(q^e-1=(q-1)\ell_e(q)\), we obtain
\(\rho\ell_h(q)\equiv1\pmod{\ell_e(q)}\).
\end{proof}

\begin{proof}[Proof of Theorem~\ref{thm:main}]
If \(q=2\), then \((-a)^{\ell_e(q)}=a^{2^e-1}=1\) for every
\(a\in\mathbb F_{2^e}^{\times}\), and also \(f(0)=f(a)=0\); hence no
permutation occurs.  Let \(q>2\).  Necessity of
\textup{(\ref{item:coprime})}--\textup{(\ref{item:norm})} follows from
Lemma~\ref{lem:elementary}, and \textup{(\ref{item:inverse})} from
Corollary~\ref{cor:quotients} and Proposition~\ref{prop:rotations}.  Conversely,
Lemma~\ref{lem:elementary} gives sufficiency.
\end{proof}

\begin{proof}[Proof of Corollary~\ref{cor:number-residues}]
For \(\gcd(h,e)=1\), set \(D=\gcd(\ell_h(q),\ell_e(q))\).  Then
\(D\mid\gcd(q^h-1,q^e-1)=q-1\), so \(q\equiv1\pmod D\), whence
\(\ell_h(q)\equiv h\pmod D\) and \(\ell_e(q)\equiv e\pmod D\).  Thus \(D\mid h\)
and \(D\mid e\), so \(D=1\); hence each such \(\ell_h(q)\) is a unit modulo
\(\ell_e(q)\).  Moreover, \(\ell_h(q)=1+q+\cdots+q^{h-1}\), so for
\(1\leqslant h<e\) one has \(0<\ell_h(q)<\ell_e(q)\), and
\(h\mapsto\ell_h(q)\) is strictly increasing.  Thus the \(\varphi(e)\)
admissible values of \(h\) give distinct unit classes modulo \(\ell_e(q)\), and
inversion preserves distinctness.  Their inverses therefore give exactly
\(\varphi(e)\) residue classes.
\end{proof}

\begin{proof}[Proof of Corollary~\ref{cor:number-functions}]
If \(e=1\), the functions are \((1+a)x^r\) on \(\mathbb F_q^\times\):
there are \(q-2\) possible scalars \(1+a\in\mathbb F_q^\times\setminus\{1\}\) and
\(\varphi(q-1)\) permutation exponent classes, and equality of two such functions forces the
same scalar and \(r\equiv r'\pmod{q-1}\).  This gives the formula since
\(\ell_1(q)=\varphi(1)=1\).

Let \(e\geqslant2\).  For \(q=2\) there are no such permutations, so assume \(q>2\).
Corollary~\ref{cor:number-residues} gives
\(\varphi(e)\) admissible unit classes modulo \(\ell_e(q)\).  The reduction map
\((\mathbb Z/(q^e-1)\mathbb Z)^\times\to(\mathbb Z/\ell_e(q)\mathbb Z)^\times\)
is surjective: given a unit class modulo \(\ell_e(q)\), choose an integer in that
class which is congruent to \(1\) modulo every prime divisor of \(q^e-1\) not
dividing \(\ell_e(q)\); the Chinese remainder theorem applies because those
primes are coprime to \(\ell_e(q)\).  Every fibre therefore has size
\(\varphi(q^e-1)/\varphi(\ell_e(q))\); hence
conditions~\textup{(\ref{item:coprime})} and \textup{(\ref{item:inverse})}
give \(\varphi(e)\varphi(q^e-1)/\varphi(\ell_e(q))\) exponent classes.
Because \(a\mapsto-a\) is a bijection, condition~\textup{(\ref{item:norm})}
excludes exactly the \(\ell_e(q)\) roots of \(z^{\ell_e(q)}=1\), leaving
\((q-2)\ell_e(q)\) coefficients.

Distinct admissible pairs give distinct functions.  Modulo \(X^{q^e}-X\), each
pair has a reduced two-term representative with
distinct nonzero exponent residues \(r\) and \(r+q-1\); the latter cannot vanish
without violating condition~\textup{(\ref{item:coprime})}.  Equality for
\((r,a)\) and \((r',a')\) therefore matches the two supports either in the same
order or in the opposite order; the latter implies \(2(q-1)\equiv0\pmod{q^e-1}\), hence
\(\ell_e(q)\mid2\), impossible since \(\ell_e(q)\geqslant q+1\geqslant4\).  Thus
\(r\equiv r'\pmod{q^e-1}\) and \(a=a'\).  Multiplying the counts proves the
formula.
\end{proof}

\section{A coprime-index extension and further directions}
\label{sec:general-d}

A natural next problem is to study the broader family
\[
 f_{r,d,a}(X):=X^r\bigl(X^{d(q-1)}+a\bigr)
 \in\mathbb F_{q^e}[X].
\]
This broader family was also studied by Hou--Pallozzi Lavorante \cite{HPL23}.
From the quotient viewpoint, however, \(d\) itself is not the most natural
organizing parameter.  Set
\[
 \eta:=\gcd\bigl(d,\ell_e(q)\bigr).
\]
On the cyclic group
\(\boldsymbol{\mu}_{\ell_e(q)}:=\{\omega\in\mathbb F_{q^e}^{\times}:\omega^{\ell_e(q)}=1\}\),
the unit group
\((\mathbb Z/\ell_e(q)\mathbb Z)^\times\) acts on the exponent residue \(d\) by
multiplication, and its orbits are determined by \(\gcd(d,\ell_e(q))\).
Equivalently, a unit-power reparametrization changes \(d\) within its fixed
\(\eta\)-class while rescaling \(r\) compatibly.  Thus \(\eta\), rather than
\(d\) alone, is the natural first invariant for organizing the general problem.

The case \(\eta=1\) is already completely determined by
Theorem~\ref{thm:main}.

\begin{corollary}[The coprime-index case]
\label{cor:general-d-coprime}
Let \(q\) be a prime power, let \(e\geqslant2\), let \(d,r\geqslant1\), and let
\(a\in\mathbb F_{q^e}^{\times}\).  Suppose that
\(\gcd(d,\ell_e(q))=1\).  Then
\(f_{r,d,a}(X)=X^r\bigl(X^{d(q-1)}+a\bigr)\) permutes
\(\mathbb F_{q^e}\) if and only if
\begin{enumerate}[label=\textup{(\roman*)},ref=\roman*]
\item \(\gcd(r,q-1)=1\);
\item \((-a)^{\ell_e(q)}\ne1\);
\item there exists \(h\) with \(1\leqslant h<e\), \(\gcd(h,e)=1\), and
\(r\ell_h(q)\equiv d\pmod{\ell_e(q)}\).
\end{enumerate}
\end{corollary}

\begin{proof}
Write \(Q=q^e\), and let
\(\Phi(X)=X^r(X^d+a)^{q-1}\in\mathbb F_Q[X]\).  For
\(x\in\mathbb F_Q^\times\), put \(y=x^{q-1}\in
\boldsymbol{\mu}_{\ell_e(q)}\).  The fibers of \(x\mapsto x^{q-1}\) are the
\(\mathbb F_q^\times\)-cosets, and
\(f_{r,d,a}(\zeta x)=\zeta^rf_{r,d,a}(x)\).  Hence \(f_{r,d,a}\) permutes
\(\mathbb F_Q\) if and only if all three of the following conditions hold:
\(\gcd(r,q-1)=1\), \(f_{r,d,a}\) has no nonzero root, and the restriction of
\(\Phi\) to \(\boldsymbol{\mu}_{\ell_e(q)}\) permutes this group.  Since
\(\gcd(d,\ell_e(q))=1\), the map \(y\mapsto y^d\) permutes
\(\boldsymbol{\mu}_{\ell_e(q)}\); consequently \(f_{r,d,a}\) has a nonzero
root exactly when \((-a)^{\ell_e(q)}=1\).

If \(\gcd(r,q-1)>1\), then \(f_{r,d,a}\) does not permute \(\mathbb F_Q\),
so condition~\textup{(i)} is necessary.  We may therefore assume
\(\gcd(r,q-1)=1\).  Choose \(\delta\) with
\(d\delta\equiv1\pmod{\ell_e(q)}\), and choose a positive integer \(R\) with
\(R\equiv r\delta\pmod{\ell_e(q)}\) and \(\gcd(R,q-1)=1\).  Such an \(R\)
exists by the Chinese remainder theorem: if a prime \(\pi\) divides both
\(q-1\) and \(\ell_e(q)\), then \(\pi\nmid r\delta\); for each remaining
prime \(\pi\mid q-1\), impose \(R\equiv1\pmod\pi\).  Let
\(\Psi(X)=X^R(X+a)^{q-1}\in\mathbb F_Q[X]\).  Since
\(dR\equiv r\pmod{\ell_e(q)}\), for every
\(y\in\boldsymbol{\mu}_{\ell_e(q)}\) we have
\[
 \Phi(y)=\Psi(y^d).
\]
Thus the restriction of \(\Phi\) to \(\boldsymbol{\mu}_{\ell_e(q)}\) permutes
this group exactly when the restriction of \(\Psi\) does.  If
\(x\in\mathbb F_Q^\times\) and \(y=x^{q-1}\), then
\[
 \bigl(x^R(x^{q-1}+a)\bigr)^{q-1}=\Psi(y).
\]
Since \(\gcd(R,q-1)=1\), the same fiber criterion and
Theorem~\ref{thm:main} show that \(X^R(X^{q-1}+a)\) permutes
\(\mathbb F_Q\) exactly when \((-a)^{\ell_e(q)}\ne1\) and
\(R\ell_h(q)\equiv1\pmod{\ell_e(q)}\) for some \(h\) with
\(1\leqslant h<e\) and \(\gcd(h,e)=1\).  Since
\(R\equiv r\delta\pmod{\ell_e(q)}\), this congruence is equivalent to
\(r\ell_h(q)\equiv d\pmod{\ell_e(q)}\).
\end{proof}

Thus the first genuinely new regime begins when \(\eta>1\): the map
\(y\mapsto y^d\) then has kernel of order \(\eta\) on
\(\boldsymbol{\mu}_{\ell_e(q)}\), so the preceding reduction to
Theorem~\ref{thm:main} is no longer available.  A natural next step is to
organize the general classification by \(\eta=\gcd(d,\ell_e(q))\) and analyze
how this \(\eta\)-to-one quotient geometry interacts with \(r\) and \(a\).

\section*{Acknowledgments}
This work was supported by the National Natural Science Foundation of China (No.~12441107).

\noindent\textbf{AI disclosure.} OpenAI's ChatGPT (GPT-5.6 Sol) was used to assist with
exploratory proof development and auditing, exploratory computations, literature
searches and reference identification, English-language editing, and LaTeX debugging.
The results and proofs presented in this paper are entirely theoretical and do not
rely on computer-assisted calculations, exhaustive computation, or computer-assisted
proof.  All mathematical arguments, results, and
conclusions were independently checked by the author. All AI-assisted outputs were
critically reviewed and edited by the author, who takes full responsibility for the
correctness and content of the paper.

\providecommand{\bysame}{\leavevmode\hbox to3em{\hrulefill}\thinspace}
\providecommand{\MR}{\relax\ifhmode\unskip\space\fi MR }
\providecommand{\MRhref}[2]{%
  \href{http://www.ams.org/mathscinet-getitem?mr=#1}{#2}
}
\providecommand{\href}[2]{#2}

\end{document}